\documentclass[11pt,reqno]{amsart}

\usepackage{a4wide}
\usepackage{amsmath,amssymb} 
\usepackage{bbm}
\usepackage{color}
\usepackage{tkz-graph}
\tikzset{
  LabelStyle/.style = { rectangle, rounded corners, draw,
                        minimum width = 1em, 
                        font =  },
  VertexStyle/.append style = { inner sep=3pt,
                                font = \large},
  EdgeStyle/.append style = {->, bend left} }

\theoremstyle{plain}
\newtheorem{theorem}{Theorem}
\newtheorem{prop}[theorem]{Proposition}
\newtheorem{lemma}[theorem]{Lemma}
\newtheorem{coro}[theorem]{Corollary}

\theoremstyle{definition}
\newtheorem{definition}[theorem]{Definition}
\newtheorem{remark}[theorem]{Remark}
\newtheorem{question}[theorem]{Question}
\newtheorem{example}[theorem]{Example}

\newcommand{\Z}{{\mathbb Z}}

\newcommand{\R}{{\mathbb R}}
\newcommand{\N}{{\mathbb N}}

\newcommand{\mc}{\mathcal}

\newcommand{\freq}{\operatorname{freq}}

\newcommand{\dsub}{\rho}

\newcommand{\exend}{\hfill $\Diamond$}

\begin{document}

\title{
Dynamical systems arising
from random substitutions
}

\author{Dan Rust}
\author{Timo Spindeler}
\date{\today}
\address{Fakult\"at f\"ur Mathematik, Universit\"at Bielefeld, \newline
\hspace*{\parindent}Postfach 100131, 33501 Bielefeld, Germany}
\email{drust@math.uni-bielefeld.de, tspindel@math.uni-bielefeld.de}

\begin{abstract}  
Random substitutions are a natural generalisation of their classical `deterministic' counterpart, whereby at every step of iterating the substitution, instead of replacing a letter with a predetermined word, every letter is independently replaced by a word from a finite set of possible words according to a probability distribution. We discuss the subshifts associated with such substitutions and explore the dynamical and ergodic properties of these systems in order to establish the groundwork for their systematic study. Among other results, we show under reasonable conditions that such systems are topologically transitive, have either empty or dense sets of periodic points, have dense sets of linearly repetitive elements, are rarely strictly ergodic, and have positive topological entropy.
\end{abstract}



\maketitle

\section{Introduction}\label{sec:intro}
Symbolic dynamical systems associated to primitive substitutions are the prototypical examples of minimal subshifts. As such, their study has been extensive \cite{bible,dhs,fog,qu} and various approaches to extending the theory have been explored, including \emph{S-adic} or \emph{mixed} systems \cite{bd,fe,gm,ru}, and systems associated to non-primitive substitutions \cite{bkm,mr}. Motivated by examples arising in physics within the study of quasicrystals, Godr\`{e}che and Luck considered the situation that the substituted image of a letter is a random variable \cite{gl}, where we now call such systems \emph{random} or \emph{stochastic}. Others have independently studied similar generalisations of substitutions, under the guise of \emph{multi-valued} or \emph{set-valued} substitutions \cite{dm}, or \emph{0L-systems} \cite{rs}. This randomised approach has recently been revisited \cite{bss,cu,de,mo,mo2,wi} with several canonical examples now being established and studied (principally via their entropy and spectrum).
In particular, Dekking has recently emphasised the need for a systematic approach in the study of random substitutions \cite{de2}.

A general theory of random substitution subshifts (hereby abbreviated to RS-subshifts) has yet to be established. It is the goal of this article to remedy this situation and provide some key topological, dynamical and ergodic theoretic results with which one is usually accustomed when studying particular classes of symbolic subshifts. Throughout this work, we establish results which are direct generalisations of well-known results appearing in the classical study of deterministic substitutions. We highlight how the situation changes when moving from the deterministic to the random situation via examples, in order to illustrate the new and more interesting phenomena. Several of these results have been established previously for particular examples and we give particular mention to the PhD thesis of Moll \cite{mo} from where several useful ideas have been borrowed.

In Section \ref{sec:background}, we outline the basic definitions of random substitutions, their associated RS-subshifts, and introduce the primary standing assumption to be considered in this work; namely the \emph{primitivity} of a substitution. Under assumption of primitivity, we establish a simple criterion in terms of the possible lengths of substituted letters for deciding when an RS-subshift is either empty or non-empty. We show that any element of the RS-subshift generates the entire RS-subshift as an orbit closure under the action of the shift and iterated substitution.

In Section \ref{sec:dynamics} we establish the key dynamical and topological properties of an RS-subshift associated to a primitive random substitution. We prove that an RS-subshift is topologically transitive by constructing an explicit element with a dense shift-orbit. RS-subshifts are in general not minimal. We show a dichotomy result for the set of periodic points with respect to the shift: the set of periodic points is either empty or dense in an RS-subshift. Due to the potentially non-trivial structure of periodic points in these subshifts, this allows for robust tools such as the Artin--Mazur zeta function to be used in the study of random substitutions, unlike in the deterministic setting where the structure of periodic points is trivial. We show that, although RS-subshifts are in general not minimal, the set of minimal subspaces is dense in the subshift---in particular, we show that the set of linearly repetitive elements of an RS-subshift is dense. As a further dichotomy result, we show that an RS-subshift is either finite or is homeomorphic to a Cantor set. As a consequence of the topological transitivity of the subshift, we show that the associated tiling space is connected.

Section \ref{sec:ergodic} is devoted to studying some measure-theoretic properties of RS-subshifts. A key tool used in establishing results is the notion of an \emph{induced} or \emph{collared} substitution. We expect that this will be a useful tool in the future study of random substitutions. The right Perron--Frobenius eigenvectors of the substitution matrices of these induced substitutions give rise to shift invariant (ergodic) measures. Moreover, we characterise those RS-subshifts which are uniquely/strictly ergodic.

In Section \ref{sec:entropy}, we provide a very mild condition under which an RS-subshift exhibits positive topological entropy, together with loose lower bounds in terms of the letter-frequencies. This opens up the study of random substitutions to similar tools developed for the study of shifts of finite type and other positive entropy subshifts, where the topological entropy is a powerful invariant. Again, this is in contrast to the deterministic setting where the entropy is always zero.

We study several key examples in Section \ref{sec:examples} which exhibit some of the more interesting behaviours described in the previous sections. We provide two very different representations of the full $2$-shift as an RS-subshift. We show that the golden shift can also be realised as an RS-subshift---this leads to the question of whether SFTs can typically be represented as such. This question will be addressed in forthcoming work \cite{grs}. We show that examples of sofic shifts of non-finite type can also be described as RS-subshifts. Of particular interest is the study of the \emph{random period doubling} substitution which shares many of the dynamical properties of a sofic shift, and which we show is not topologically mixing. We conclude with a list of open questions motivated by the results of the previous sections.

\section{Random substitutions}\label{sec:background}
An \emph{alphabet} $\mc A = \{a_1, \ldots, a_n\}$ is a finite set of symbols $a_i$ referred to as \emph{letters}. A \emph{word} $u$ in $\mc A$ is a finite concatenation of letters $u = a_{i_1}\cdots a_{i_{\ell}}$ and we let $|u| = \ell$ denote the \emph{length} of $u$. The \emph{empty word} $\epsilon$ is the unique word of length $0$. We let $|u|_{a_i} = \ell_i$ denote the number of times $\ell_i$ that the letter $a_i$ appears in the word $u$. We let $\mc A^\ell$ denote the set of length-$\ell$ words in $\mc A$ and we let $\mc A^+ = \bigcup_{\ell \geq 1} \mc A^\ell$ denote the set of all finite words in $\mc A$ with positive length. If $\epsilon$ is also considered, then we write $\mc A^\ast = \mc A^+ \cup \{\epsilon\}$. The \emph{concatenation} $uv$ of two words $u = a_{i_1} \cdots a_{i_\ell}$ and $v = a_{j_1} \cdots a_{j_m}$ is given by $uv = a_{i_1} \cdots a_{i_\ell} a_{j_1} \cdots a_{j_m}$. We let $\mc A^\Z = \{\cdots a_{i_{-1}}a_{i_0}a_{i_1}\cdots \mid a_{i} \in \mc A\}$ denote the set of \emph{bi-infinite sequences} in $\mc A$ and endow $\mc A^\Z$ with the product topology, where $\mc A$ is a finite discrete space. For $w \in \mc A^\Z$ and $i \leq j$, let $w_{[i,j]}$ denote the finite word $w_{[i,j]} := w_{i} w_{i+1} \cdots w_{j-1} w_j$.

Let $\mc P (Y)$ be the \emph{power set} of $Y$. By an abuse of notation, we assign a \emph{set-valued function} $f \colon X \to Y$ to a function $f \colon X \to \mc P (Y)$, where the distinction between functions and set-valued functions will always be clear by context. If the set $f(x)$ is finite for all $x \in X$ then we call $f$ a \emph{finite-set-valued function}.

\begin{definition}\label{Def:1}
  Let $\mc A = \{a_1, \ldots, a_n\}$ be a finite alphabet. A \emph{random substitution} (or \emph{stochastic substitution}) $(\vartheta, \boldsymbol{P})$ is a finite-set-valued function $\vartheta \colon \mc A \to \mc A^+$ together with a set of probability vectors
  \[
  \boldsymbol{P} = \left\{{\boldsymbol p}_i=(p_{i1},\ldots,p_{ik_i}) \mid {\boldsymbol p}_i\in\, [0,1]^{k_i}\  \;
  \text{ and }\; \sum_{j=1}^{k_i}p_{ij}=1,\ 1\leq i\leq n\right\}
\] 
   for $k_1, \ldots, k_n \in \N \setminus \{0\}$ such that
\[
  \vartheta \colon \; a_i\mapsto 
\begin{cases}
  w^{(i,1)}, & \text{with probability } p_{i1}, \\
  \quad \vdots  & \quad \quad \quad \quad \ \vdots \\
  w^{(i,k_i)}, & \text{with probability } p_{ik_i},
\end{cases} 
\]  
  for $1\leq i\leq n$, where each $w^{(i,j)} \in \mc A^+$. The corresponding 
  \emph{random substitution matrix} is defined by
\[
    M_{(\vartheta,\boldsymbol{P})}\, := \left[\sum_{q=1}^{k_j} p_{jq} 
   |w^{(j,q)}|_{a_i} 
   \right]_{1\leq i,j\leq n}.
\]
If the set of words $\vartheta(a_i)$ is a singleton for each $a_i \in \mc A$, then we call $\vartheta$ \emph{deterministic}. If $p_{ij} \neq 0$ for all values of $i,j$ then we say that $(\vartheta,\boldsymbol{P})$ is \emph{non-degenerate}; otherwise, we say that $(\vartheta,\boldsymbol{P})$ is \emph{degenerate}.
\end{definition}

Less formally, a random substitution $\vartheta$ assigns to every letter $a_i \in \mc A$, a finite set of words $w^{(i,j)}$ in $\mc A$ with a corresponding probability $p_{ij}$. There is a natural extension to random substitutions that assign possibly infinitely many words to the letters of $\mc A$, but we only consider the finite-set-valued case in the present article (with the exception of a single example in Section \ref{sec:examples}). Most of our results extend to the infinite-valued setting, but not all.
\begin{remark}
Our statements will often be independent of the random variables $p_{ij}$ in the non-degenerate setting and so we may suppress them in the notation and write a random substitution more compactly as 
\[
\vartheta \colon a_i \mapsto \{w^{(i,1)}, \ldots, w^{(i,k_i)}\},
\]
where we implicitly assume that $\vartheta$ is non-degenerate and we can suppress the pair notation that includes $\boldsymbol{P}$.	 This probability-independent setting coincides with the definition of a \emph{multi-valued substitution} in the literature \cite{cu,dm}, where a multi-valued substitution is simply a finite-set-valued function $\vartheta \colon \mc A \to \mc A^+$. For the sake of clarity however, we will exclusively use the term `random substitution', even in such cases.	\exend
\end{remark}

As in the classical setting for deterministic substitutions, there is an obvious extension of the action of a random substitution to the set of finite words $\mc A^+$ and to the set of bi-infinite sequences $\mc A^\Z$. The image of a word $u \in \mc A^+$ under the action of a random substitution $\vartheta$ is a finite set of words $\vartheta(u) = \{u^{(1)}, \ldots, u^{(m)}\}$ where each $u^{(i)}$ is given by applying $\vartheta$ independently to every letter in $u$ and concatenating the images in the order prescribed by $u$. As a basic example, consider the substitution $\vartheta \colon \mc A \to \mc A^+ $ on the alphabet $\mc A = \{a,b\}$ given by
\[
  \vartheta\colon
  \begin{cases}
      a\mapsto 
      \begin{cases}
      ab, & \text{with probability } p,      \\
      ba, & \text{with probability } 1-p,
      \end{cases} \\
      b\mapsto aa, \quad\:\:\:\text{with probability }1.
  \end{cases} 
\]
Then $\vartheta(aba)$ is the set of words $\{abaaab,abaaba,baaaab,baaaba\}$ with probabilities $p^2$, $p(1-p)$, $(1-p)p$, $(1-p)^2$ respectively. We may then consider finite powers of the random substitution $\vartheta^k$ so that, in the above example for instance, $\vartheta^2(a) = \vartheta(\{ab,ba\}) = \{abaa,baaa,aaab,aaba\}$ (with corresponding probabilities).
The extension to $\mc A^\Z$ is similar.

Consequently, $\vartheta(a_i)$, and hence 
$\vartheta(u)$ for any $u \in \mc A^+$, should be considered as a random 
variable with finitely many possible realisations. The $(i,j)$-entry of the substitution matrix $M_\vartheta$ is then the expected number of times\footnote{Properly, the matrix $M_{(\vartheta,\boldsymbol{P})}$ should itself be considered as a random variable, where the definition of $M_{(\vartheta,\boldsymbol{P})}$ introduced in Definition \ref{Def:1} is then the expected value of that random variable. For our purposes, this formal consideration is unnecessary.} that the letter $a_j$ appears in the word $\vartheta(a_i)$.

To proceed, we need to introduce generalised notation from symbolic dynamics that is particular to the random setting.

\begin{definition}
Let $u \in \mc A^+$ and $v \in \mc A^+$ or $v \in \mc A^\Z$. By $u \triangleleft v$ we mean that $u$ is a subword of $v$, and by $u \blacktriangleleft \vartheta^k(v)$ we mean 
that $u$ is a subword of at least one image of $v$ under $\vartheta^k$ for 
some $k \in \N$. Similarly, by $u \overset{\bullet}{=} \vartheta^k(v)$ we mean 
that there is at least one image of $v$ under $\vartheta^k$ that coincides with 
$u$. We say that $u$ is a \emph{realisation} of $\vartheta^k(v)$.
\end{definition}

\begin{definition}
A random substitution $\vartheta \colon \mc A^+ \to \mc A^+$ is called \emph{irreducible}
if for each pair $(i,j)$ with $1\leq i,j\leq n$, there is a $k\in\N$ such that $a_i
\blacktriangleleft\vartheta^k(a_j)$. Moreover, $\vartheta$ is called \emph{primitive}
if there is a $k\in\N$ such that for all 
$1\leq i,j\leq n$ we have $a_i\blacktriangleleft\vartheta^k(a_j)$.
\end{definition}

\begin{remark}
As in the deterministic case, a non-degenerate random substitution $\vartheta$ is irreducible/primitive 
if and only if $M_{\vartheta}$ is an irreducible/primitive matrix. Note that degenerate random substitutions can be irreducible/primitive and have non-irreducible/non-primitive substitution matrix.
\exend
\end{remark}

\begin{definition}
A word $u \in \mc A^+$ is called ($\vartheta$-)\emph{legal} if there is a $k \in \N$ such 
that $u \blacktriangleleft \vartheta^k(a_j)$ for some $j \in \{1, \ldots, n\}$. We define the 
\emph{language} of $\vartheta$ by
\[
\mc L_{\vartheta}:=\{u \in \mc A^* \mid u \text{ is } \vartheta\text{-legal}\}.
\]
If $w \in \mc A^+$ or $w \in \mc A^{\Z}$, we define the language of $w$ by
\[
\mc L(w):=\{u \in \mc A^* \mid u \triangleleft w\}.
\]
If $X \subseteq \mc A^\Z$, we define the language of $X$ by 
\[
\displaystyle{\mc L(X) := \bigcup_{w \in X} \mc L(w)}.
\]
We let $\mc L_{\vartheta}^{\ell} \subseteq \mc L_{\vartheta}, \mc L^{\ell}(w) \subseteq \mc L(w)$ and $\mc L^{\ell}(X) \subseteq \mc L(X)$ denote the set of elements of length $\ell$.

Let $(X_\vartheta,S)$ denote the \emph{random substitution subshift} (\emph{RS-subshift} for short) associated with the random substitution $\vartheta$, where
\[
X_{\vartheta}:= \{w\in\mc A^{\Z} \mid \mc L(w) \subseteq \mc L_{\vartheta}\}
\]
and $S \colon X_\vartheta \to X_\vartheta$ is the usual shift operator defined by $S(w)_i = w_{i+1}$ for an element $w \in X_\vartheta$.
\end{definition}

\begin{remark}
The language and the RS-subshift are independent of the explicit values of $p_{ij}$, assuming the substitution is non-degenerate.   \exend
\end{remark}

As $X_\vartheta$ is defined in terms of a language, it is immediately a closed (hence compact) shift invariant subspace of $\mc A^\Z$, and so $X_\vartheta$ is a subshift.
Unlike in the deterministic case, it is not true that a primitive random substitution $\vartheta$ always has non-empty RS-subshift $X_\vartheta$. Take as an example the primitive substitution $\vartheta \colon a \mapsto \{a,b\}, b\mapsto \{a\}$ whose language $\mc L_\vartheta = \{a,b\}$ is finite and so has empty RS-subshift. This example also highlights the fact that we can only say that $\mc L(X_\vartheta) \subseteq \mc L_\vartheta$ in general, in contrast to the primitive deterministic case where we always have equality. Thankfully, an RS-subshift is empty only in very specific circumstance and we can characterise the non-degenerate primitive random substitutions whose corresponding RS-subshift is non-empty.

\begin{prop}
Let $\vartheta$ be a primitive random substitution. Then, the RS-subshift $X_{\vartheta}$ is empty if and only if the set of realisations of $\vartheta(a_i)$ consists only of words of length $1$, for every $a_i\in\mc A$. 
\end{prop}
\begin{proof}
If the set of realisations of $\vartheta(a_i)$ consists only of words of length $1$, for every $a_i$, it is immediate that the subshift is empty.

On the other hand, assume that there is a letter, say $a_k$, such that at least one
realisation of $\vartheta(a_k)$ has length at least $2$. Then, there is a $N\in \N$ such that $a_k$ is a subword of $\vartheta^N(a_i)$ for all $a_i\in\mc A$. Hence, $\vartheta^{N+1}(a_i)$ has at least length $2$. Arguing the same lines, we obtain that $\vartheta^{2N+2}(a_i)$ has length at least $4$. Inductively, we find that $\vartheta^{kN+k}(a_i)$ has length at least $2^k$. As $\mc A$ is finite, this implies that the subshift is non-empty. 
\end{proof}

From now on, we assume that $X_\vartheta$ is always non-empty. Many of the proofs that follow will require picking an element in $X_\vartheta$ and so this assumption will be implicit.

\begin{lemma}\label{lem:closed}
If $w \in X_\vartheta$, then $\vartheta(w) \in X_\vartheta$ for every realisation of $\vartheta(w)$.
\end{lemma}
\begin{proof}
Let $w \in X_\vartheta$. Suppose $u \in \mc A^\ast$ is a subword of a realisation $\hat{w} \overset{\bullet}{=} \vartheta(w)$. It follows that there exists a subword $v$ of $w$ such that $u \blacktriangleleft \vartheta(v)$. As $v$ is a subword of $w$ and $w\in X_\vartheta$, there exists a natural number $k \in \N$ and a letter $a\in \mc A$ such that $v \blacktriangleleft \vartheta^k(a)$. It follows that $u \blacktriangleleft \vartheta^{k+1}(a)$. As $u$ was chosen to be an arbitrary subword of $\hat{w}$, it follows that $\hat{w} \in X_\vartheta$.
\end{proof}

It is a classic result in the deterministic setting \cite[Prop. 5.3]{qu} that if $\dsub$ is a primitive deterministic substitution and $w_0$ is a fixed point of some power of $\dsub$, one 
can alternatively define the subshift by
\[
X_{\dsub} := \overline{\{S^k w_0 \mid k \in \Z\}},
\]
where here we let $\overline{A} \subseteq X$ denote the closure of the subset $A$ in the space $X$. Even more, we have $X_{\dsub}=\overline{\{S^k w \mid k\in\Z\}}$ for every $w \in X_{\dsub}$. In the random situation, there is no direct analogue of a fixed point. However, one does have the following result of a similar flavour.

\begin{prop}
Let $\vartheta$ be a primitive random substitution with RS-subshift $X_\vartheta$. Let $w$ be any element of $X_\vartheta$. Then
\[
X_\vartheta = \overline{\{S^k(x) \mid x \overset{\bullet}{=} \vartheta^n(w), k \in \Z, n \geq 0 \}}
\]
where $\vartheta^n$ is understood to range over all possible realisations.
\end{prop}
\begin{proof}
Let $w \in X_\vartheta$ be fixed.
Let $A = \{S^k(x) \mid x \overset{\bullet}{=} \vartheta^n(w), k \in \Z, n \geq 0 \}$.
We first show the left-to-right inclusion. Let $x \in X_\vartheta$ be an arbitrary element. By primitivity, and the fact that $x_{[-\ell,\ell]}$ is a legal word, let $n_\ell$ be the least natural number such that $x_{[-\ell,\ell]} \blacktriangleleft \vartheta^{n_\ell}(w_0)$. Then there exists a $k_\ell$ such that $S^{k_\ell}(\vartheta^{n_\ell}(w))_{[-\ell,\ell]} \overset{\bullet}{=} x_{[-\ell,\ell]}$. Hence, there exists a realisation $x^{(\ell)}$ of $S^{k_\ell}(\vartheta^{n_\ell}(w))$ such that $x^{(\ell)}_{[-\ell,\ell]} = x_{[-\ell,\ell]}$ and clearly $x^{(\ell)} \in A$. By construction, $\lim_{\ell \to \infty} x^{(\ell)} = x$ and so $x \in \overline{A}$. It follows that $X_\vartheta \subseteq \mc \overline{A}$.

The right-to-left inclusion is immediate as Lemma \ref{lem:closed} and the shift-invariance of $X_\vartheta$ automatically gives us $A \subseteq X_\vartheta$ and so
\[
\overline{A} \subseteq \overline{X_\vartheta} = X_\vartheta
\] 
by compactness of $X_\vartheta$. It follows that $X_\vartheta = \overline{A}$.
\end{proof}
Let 
\[
r\colon \Z \to (\mc A^+)^n,\quad \ \ i\mapsto \big(\vartheta^{(i)}(a_1),\ldots,\vartheta^{(i)}(a_n) \big) 
\]
be one particular bi-infinite realisation of the random substitution $\vartheta$, where $\vartheta^{(i)}$ denotes the realisation at index $i$. Then, we define the map
\[
\vartheta[r]\colon \mc A^{\Z} \to \mc A^{\Z}, \quad \ \ w=(w_i)_{i\in\Z} \mapsto \big( \vartheta^{(i)}(w_i) \big)_{i \in \Z}.
\] 
The continuity of $\vartheta[r]$ follows as in the deterministic case. Now, let $w \in X_{\vartheta}$ be arbitrary. We want to show that elements have pre-images with respect to the random substitution in the RS-subshift. Finding pre-images relies on the existence of a particular choice of a concrete realisation which means that we can make use of the continuity of $\vartheta[r]$.
\begin{lemma}\label{lem:pre-images}
Let $\vartheta$ be a primitive random substitution. If $w \in X_\vartheta$, then there exists an element $w' \in X_{\vartheta}$ and $n \geq 0$ with $w \overset{\bullet}{=} S^n(\vartheta(w'))$.
\end{lemma}
\begin{proof}
Recall that all finite subwords of $w$ are $\vartheta$-legal. This means that there is a $k \in \N$ such that for all $\ell \in \N_0$ the centrally positioned subword $w_{[-\ell,\ell]}$ of $w$ is a subword of $\vartheta^k(a_{\ell})$ for some $a_{\ell} \in \mc A$ and we choose $k$ minimal with this property. This implies that there is some image $v_{\ell} \overset{\bullet}{=} \vartheta^{k-1}(a_{\ell})$ with $w_{[-\ell,\ell]} \blacktriangleleft \vartheta(v_{\ell})$ and we choose the indexing of $v_{\ell}$ such that $S^{n_\ell}(\vartheta(v_{\ell}))$ has its subword $w_{[-\ell,\ell]}$ centrally positioned around the reference point for the minimal non-negative choice of the shift index $n_\ell \geq 0$. Now, let $\hat{w}^{(\ell)}$ be an element of $X_\vartheta$ such that $\hat{w}^{(\ell)}_0 = a_{\ell}$---this is certainly possible for primitive $\vartheta$. Then there again exists a shift $S^{m_\ell}$ of a realisation of $\vartheta^{k-1}(\hat{w}^{(\ell)})$ such that $w^{(\ell)} \overset{\bullet}{=} S^{m_\ell}(\vartheta^{k-1}(\hat{w}^{(\ell)})) \in X_\vartheta$ contains a centrally positioned copy of the word $v_\ell$ with the necessary indexing as described. In particular, $S^{n_\ell}(\vartheta[r_\ell](w^{(\ell)}))_{[-\ell,\ell]} = w_{[-\ell,\ell]}$ for a particular realisation $\vartheta[r_\ell]$. As the possible lengths of the words $\vartheta(a_i)$ are bounded, the possible values of the minimal non-negative shift indices $n_\ell$ are also bounded, and so we can pick an infinite subsequence of the lengths $\ell(j)$ such that $n_{\ell(j)} = n$ is constant.

Due to the compactness of $X_{\vartheta}$, there is a further subsequence $(\ell(j)_i)_{i\in\N}$ such that the limit
$w^{(\infty)} := \lim_{i\to\infty} w^{(\ell(j)_i)}$ exists. Note that, by construction, any two realisations $\vartheta[r_{\ell(j)_i}]$ and $\vartheta[r_{\ell(j)_{i+m}}]$ agree with each other on all blocks of radius $\ell(j)_i$ around the $0$th index, and so there is a natural limit realisation of the random substitution $\vartheta[r]$, given by letting $\ell(j)_i \to \infty$, which can be used in place of each of the realisations $\vartheta[r_{\ell(j)_i}]$. Putting things together yields
\[
w = \lim_{i\to\infty} S^n(\vartheta[r] (w^{(\ell(j)_i)})) = S^n(\vartheta[r](w^{(\infty)}))
\]
by continuity of the shift and $\vartheta[r]$.
\end{proof}
As one can quickly see, the pre-image of any such $w$ is in general not uniquely defined. This is a further difference to the purely deterministic primitive case where this uniqueness holds for aperiodic subshifts by the celebrated result of Moss\'{e} \cite{mos} who showed the equivalence of aperiodicity with the property known as \emph{unique recognisability}; we refer to \cite[Sec. 5.5.2]{qu} and \cite[Sec. 7.2.1]{fog} for background.

\section{Dynamics and Topology}\label{sec:dynamics}
In the deterministic setting it is immediate that the subshift $(X_{\dsub},S)$ of a primitive
substitution $\dsub$ contains dense orbits. Even more, the subshift is minimal; see
\cite[Prop. 5.5]{qu}. We will now show that the first statement remains true for random
substitution, while the second, in general, is false.

\begin{prop} \label{prop:top_tran}
Let $\vartheta$ be a primitive random substitution. The RS-subshift $(X_{\vartheta},S)$ contains an element with dense shift-orbit in $X_\vartheta$.
\end{prop}
\begin{proof} 
Let $w \in X_\vartheta$ and let $a_i \in \mc A$ be a letter which appears infinitely often in $w$. As $\mc A$ is finite and $w$ has infinitely many entries, such an $a_i$ must exist. Let $\mc G_k := \{u \in \mc A^+ \mid u \overset{\bullet}{=}\vartheta^{k}(a_i)\}$.
For all $k \geq 1$, let $w^{(-k)}$ be an element in $X_\vartheta$ such that there exists an ${n_k} \geq 0$ with $w \overset{\bullet}{=} S^{n_k}(\vartheta^k(w^{(-k)})$. Such a sequence of elements exists by a repeated application of Lemma \ref{lem:pre-images}.

 Define the sequence $(w^{(k)})_{k \in \N_0}$ by $w^{(0)} = \vartheta(w)$ and $w^{(k)} \overset{\bullet}{=} \vartheta^{k+1}(w^{(-k)})$ such that the realisation of $\vartheta^{k+1}$ is chosen so that the first $|\mc G_k|$ $a_i$s, closest to the origin in $\vartheta^{k}(w^{(-k)})$, are bijectively mapped by $\vartheta^k$ onto the words in $\mc G_k$ and there exists an $r_k \geq 0$ such that the words $w^{(k)}_{[-r_k,r_k]}$ and $w^{(k-1)}_{[-r_k,r_k]}$ agree, and contain all words in $\mc G_k$. This is possible because of how the elements $w^{(-k)}$ are defined and our ability to choose the realisations of the substitution carefully. By Lemma \ref{lem:closed}, $w^{(k)} \in X_{\vartheta}$ for all $k \geq 0$. Furthermore, due to the primitivity of $\vartheta$, we find that for every $\ell \in \N$ there is a $k \in \N$ such that for every $u \in \mc L_{\vartheta}^{\ell}$, there exists a $v \in \mc G_k$ such that $u \triangleleft v$. In particular, for every $\ell \in \N$ there is a $k \geq 0$ such that $\mc L_{\vartheta}^{\ell} \subseteq \mc L(w^{(k)})$. Since $X_{\vartheta}$ is compact, there is a subsequence $(w^{(k_i)})_{i \in \N}$ converging towards a word $w^{(\infty)} \in X_{\vartheta}$. Also, as the words $w^{(k)}$ agree on larger and large patches around the origin as $k$ grows, it means that $\mc L_{\vartheta}^{\ell} \subseteq \mc L(w^{(\infty)})$ for every $\ell \in \N$ (that is, no legal words are `pushed to infinity' in the limit). This implies
\[
\overline{\{S^kw^{(\infty)} \mid k\in\Z\}} = X_{\vartheta}.
\]
\end{proof}

The bi-infinite word $w^\infty$ constructed in the proof of Proposition \ref{prop:top_tran} is somewhat of an analogue to a fixed point, but for the random setting.

To see that $(X_{\vartheta},S)$ is not necessarily minimal (although $\vartheta$ is primitive), consider the following random substitution first studied by Godr\`{e}che and Luck \cite{gl}. Let $\mc A=\{a,b\}$, $\boldsymbol{p_1}=(p,q)$ and $\boldsymbol{p_2}=
(1)$. The \emph{random Fibonacci substitution} $\sigma$ is defined by
\[
  \sigma\colon \; 
  \begin{cases}
      a\mapsto 
      \begin{cases}
      ba, & \text{with probability } p,      \\
      ab, & \text{with probability } q,
      \end{cases} \\
      b\mapsto a .
  \end{cases} 
\] 
If we denote by $X_{\text{Fib}}$ the well-known minimal Fibonacci subshift, it is not difficult to see
that $X_{\text{Fib}}$ is a proper subset of $X_{\sigma}$ (there are elements in $X_{\sigma}$ 
which have $bb$ as a subword). Hence, $(X_{\sigma},S)$ cannot be minimal.

It would be interesting to know whether there are sufficient or necessary conditions for $(X_{\vartheta},S)$ to be minimal. We shall address this question in Section \ref{sec:ergodic}.

\begin{definition}
For a continuous dynamical system $X = (X,f)$, let
\[
\operatorname{Per}(X) = \{x \in X \mid \exists n > 0, \text{ such that } f^n(x) = x\}
\] denote the set of periodic points of $X$ under the action of $f$.
\end{definition}
The following is a surprising dichotomy result for periodic points in RS-subshifts.
\begin{prop}\label{prop:periodic}
Let $\vartheta$ be a primitive random substitution with RS-subshift $(X_\vartheta,S)$. Either $\operatorname{Per}(X_\vartheta)$ is empty or $\operatorname{Per}(X_\vartheta)\subseteq X_\vartheta$ is dense.
\end{prop}
\begin{proof}Suppose that $\operatorname{Per}(X_\vartheta)$ is non-empty. Let $w^{(0)}$ be a periodic point of $X_\vartheta$ with prime period $p \geq 1$. So $S^{p}(w^{(0)}) = w^{(0)}$ and $S^i(w^{(0)}) \neq w^{(0)}$ for all $0 < i < p$. 

Let $w \in X_\vartheta$. Let $a = w^{(0)}_0 \in \mc A$ be fixed. By primitivity, for every finite subword $u \triangleleft w$, there exists a natural number $n \geq 0$ such that $u \blacktriangleleft \vartheta^n(a)$. Let $n(u)$ be the minimal such $n$. For all $\ell \geq 1$, we have $w_{[-\ell,\ell]} \blacktriangleleft \vartheta^{n(w_{[-\ell,\ell]})}(a)$. In particular, we have
\[
w_{[-\ell,\ell]} \blacktriangleleft \vartheta^{n(w_{[-\ell,\ell]})}(w^{(0)}_0 \cdots w^{(0)}_p)
\]
and so defining $\tilde{w}^{(\ell)} := \vartheta^{n(w_{[-\ell,\ell]})}(w^{(0)})$ (in such a way that every periodic block is substituted identically) gives us $w_{[-\ell,\ell]} \blacktriangleleft \tilde{w}^{(\ell)}$.

There exist finite shifts of the sequence $\tilde{w}^{(\ell)}$ such that $S^i(\tilde{w}^{(\ell)})_{[-\ell,\ell]} = w_{[-\ell,\ell]}$. Let $w^{(\ell)}$ be the required shifted sequences. It follows that $\lim_{\ell \to \infty}(w^{(\ell)}) = w$. By construction, for all $\ell \geq 0$, the sequence $w^{(\ell)}$ is legal and periodic and so it follows that $\operatorname{Per}(X_\vartheta)$ is a dense subset of $X_\vartheta$.
\end{proof}

There exist examples exhibiting both kinds of behaviour. Any primitive aperiodic deterministic substitution is periodic-point free. Moreover, the random Fibonacci substitution introduced above has no periodic points as the relative frequencies of $a$s to $b$s is irrationally related for every element of the RS-subshift. On the other hand, as we shall see in Section \ref{sec:examples}, there exist examples of primitive random substitutions $\vartheta$ whose RS-subshift $X_\vartheta$ has $\operatorname{Per}(X_\vartheta)$ contained as a proper dense subset.

Periodic elements form a special subclass of a more general family of bi-infinite sequences; the linearly repetitive sequences.

\begin{definition}
We say that a sequence $w$ is \emph{linearly repetitive} if there exists a real number $L \geq 1$ such that for all $n \geq 1$, we have $u \in \mc L^n(w)$ and $v \in \mc L^{Ln}(w)$ implies that $u \triangleleft v$. If $w$ is a linearly repetitive sequence then we call the subshift $X_w$ \emph{linearly recurrent}.
\end{definition}
Clearly every periodic sequence with prime period $p$ is linearly repetitive with $L = p$.
It is well-known \cite{dhs} that if $\dsub$ is a primitive deterministic substitution then every $w \in X_{\dsub}$ is linearly repetitive and hence $X_{\dsub}$ is linearly recurrent. Let $\operatorname{Lin}(X_\vartheta) \subseteq X_\vartheta$ denote the set of linearly repetitive elements of $X_\vartheta$.

Although not every primitive RS-subshift contains periodic elements, they all contain linearly repetitive elements. Moreover, we have the following density result for $\operatorname{Lin}(X_\vartheta)$.

\begin{prop}\label{prop:lin}
Let $\vartheta$ be a primitive random substitution with RS-subshift $X_\vartheta$. The set $\operatorname{Lin}(X_\vartheta) \subseteq X_\vartheta$ is dense.
\end{prop}
\begin{proof}
Let $w \in X_\vartheta$. Let $a \in \mc A$ be a fixed letter. Let $n(\ell)$ be the minimal natural number such that $w_{[-\ell,\ell]} \blacktriangleleft \vartheta^{n(\ell)}(a)$ which exists by primitivity. There is then a primitive deterministic substitution $\dsub^{(\ell)} \colon \mc A \to \mc A^+$ such that $w_{[-\ell,\ell]} \triangleleft \dsub^{(\ell)}(a)$ which is a particular realisation of $\vartheta^{n(\ell)}$. In particular, $X_{\dsub^{(\ell)}} \subseteq X_\vartheta$, $X_{\dsub^{(\ell)}}$ is non-empty, and every element of $X_{\dsub^{(\ell)}}$ is linearly repetitive.

By construction, there exists a point $w^{(\ell)} \in X_{\dsub^{(\ell)}}$ such that $w^{(\ell)}_{[-\ell,\ell]} = w_{[-\ell,\ell]}$. It follows that $\lim_{\ell \to \infty} w^{(\ell)} = w$. Hence the set of linearly repetitive elements of $X_\vartheta$ is dense.
\end{proof}

Although Proposition \ref{prop:lin} only refers to the subset of linearly repetitive elements, the construction used in the proof shows that the set of substitutive sequences\footnote{By substitutive sequence, we mean an element of a subshift associated a primitive deterministic substitution.} is also dense in $X_\vartheta$. The result also trivially implies that the set of repetitive elements form a dense subset of $X_\vartheta$.

Recall that a topological space is \emph{perfect} if it contains no isolated points. For a point $x \in X$, we let $\mc O x$ denote the orbit of $x$ under the homeomorphism $f$. That is, $\mc O x = \{f^n(x) \mid n \in \Z\}$.

%

\begin{prop}\label{prop:cantor}
Let $\vartheta$ be a primitive random substitution. The RS-subshift $X_\vartheta$ is either finite or is homeomorphic to a Cantor set.
\end{prop}
\begin{proof}
We use the fact that a topological space $X$ is homeomorphic to a Cantor set if and only if $X$ is a compact, perfect, totally disconnected, metrisable space. It is clear that $X_\vartheta$ is a compact metric space. A subspace of a totally disconnected space is totally disconnected, and so $X_\vartheta$ is totally disconnected because it is a subspace of $\mc A^\Z$. It only remains to show that $X_\vartheta$ is perfect.

From Proposition \ref{prop:top_tran} we know that $X_\vartheta$ contains an element $w_0$ such that $\mc O w_0$ is a dense orbit. For all $w \in X_\vartheta \setminus \mc O w_0$, the point $w$ cannot be isolated, as all neighbourhoods $U$ of $w$ intersect $\mc O w_0$; that is, $(U \setminus \{x\}) \cap \mc O w_0 \neq \emptyset$. The only other possibility is that $\mc O w_0$ contains isolated points.

Suppose $w_0$ is linearly repetitive. Then $X_\vartheta$ is linearly recurrent, hence minimal and there exists a primitive deterministic substitution $\dsub$ such that $X_\vartheta = X_\dsub$ (this follows from the proof of Proposition \ref{prop:lin}). If $\dsub$ is deterministic, then the result is well-known \cite[Proposition 4.5]{bible}, so we may assume that $w_0$ is not linearly repetitive. In particular, $X_\vartheta$ must be infinite.

Recall from Proposition \ref{prop:lin} that the set $\operatorname{Lin}(X_\vartheta)$ is a dense subset of $X_\vartheta$. In particular, for every point $w \in \mc O w_0$ and every neighbourhood $U$ of $w$, we have $U \cap \operatorname{Lin}(X_\vartheta) \neq \emptyset$, and as $w \notin \operatorname{Lin}(X_\vartheta)$, we can conclude that $w$ is not isolated. It follows that $X_\vartheta$ is perfect and hence is homeomorphic to a Cantor set.
\end{proof}
Recall that a topological dynamical system $(X,f)$ is \emph{topologically transitive} if for all open sets $U,V\subseteq X$, there exists a natural number $n \geq 0$ such that $f^n(U) \cap V \neq \varnothing$. If $X$ is either a Cantor set or a finite discrete space, then $(X,f)$ is topologically transitive if and only if $(X,f)$ admits a dense orbit (this does not hold in general \cite{dk}). Hence, transitivity of primitive RS-subshifts follows from a simple application of Proposition \ref{prop:top_tran} and Proposition \ref{prop:cantor}.
\begin{coro}\label{cor:top_tran}
Let $\vartheta$ be a primitive random substitution. The RS-subshift $(X_\vartheta,S)$ is topologically transitive.
\end{coro}

Although tiling spaces are not the main focus of this work, we briefly mention them in order to generalise a result attributed to G\"{a}hler and Miro in the context of the random Fibonacci substitution \cite{gp}. They showed that the tiling space associated to the random Fibonacci substitution is connected. In fact, we can show that this is true for any primitive random substitution.

\begin{definition}
Let $(X,s)$ be a subshift. Let $\Omega_X$ denote the \emph{tiling space} or \emph{continuous hull} associated with $X$ defined by
\[
\Omega_X = X \times [0,1]/\sim
\]
where $(x,1) \sim (Sx,0)$ generates an equivalence relation that glues together the ends of intervals associated with elements $x \in X$ with the beginning of the intervals associated with the shifted element $Sx \in X$.

If $\vartheta$ is a random substitution with RS-subshift $X_\vartheta$ then we write $\Omega_\vartheta := \Omega_{X_\vartheta}$ and call $\Omega_\vartheta$ the \emph{tiling space} associated with $\vartheta$. 
\end{definition}
For an accessible introduction to the study of tiling spaces from a topological perspective, we recommend the book of Sadun \cite{sa}.
\begin{prop}
Let $X$ be a subshift admitting a dense orbit. The associated tiling space $\Omega_X$ is connected.
\end{prop}
\begin{proof}
We first note that for every $x \in X$, the set $x + \R := \{(S^n(x),t) \mid n \in \Z, t \in [0,1)\}$ is a path connected subspace of $\Omega(X)$ which is the continuous image of a real line, given by the mapping $r \mapsto (S^{\lfloor r \rfloor}x,r-\lfloor r \rfloor)$.

Let $w \in X$ be an element with a dense orbit. Suppose for contradiction that $\Omega_X$ is not connected. Then $\Omega_X = U \cup V$ for some disjoint non-empty clopen subsets $U$ and $V$. It follows that $U \cap (w + \R)$ and $V \cap (w + \R)$ are disjoint clopen subsets of $\mc O w$ with the subspace topology and $(U \cap (w + \R))\cup (V \cap (w + \R)) = \mc O w$. It remains to show that both $U \cap (w + \R)$ and $V \cap (w + \R)$ are non-empty.

Let $(x,t) \in U$ which exists as $U$ is non-empty. As $w$ has dense orbit in $X$, and $U$ is open, there exists an $n \in \Z$ such that $(S^n w,t) \in U$. It follows that $(S^n w,t) \in U \cap (w + \R)$ and so $U \cap (w + \R)$ is non-empty. Similarly, $V \cap (w + \R)$ is non-empty and so $w + \R$ is not connected. This contradicts the fact that $w + \R$ is the continuous image of a connected space and so must be connected. It follows that $\Omega_X$ is connected.
\end{proof}
Then, by a simple application of Proposition \ref{prop:top_tran}, the result follows as a corollary.
\begin{coro}
Let $\vartheta$ be a primitive random substitution. The associated tiling space $\Omega_\vartheta$ is connected. \qed
\end{coro}
G\"{a}hler and Miro also studied the first \v{C}ech cohomology group $\check{H}^1(\Omega_\vartheta)$ of the tiling space in the particular case where $\vartheta$ is the random Fibonacci substitution \cite{gp}. They showed that $\check{H}^1$ has infinite rank. Using results from this section, one can show the same in much more generality. In fact, a tiling space associated with a primitive random substitution is either minimal or has its first \v{C}ech cohomology group of infinite rank. As cohomology of tiling spaces is not the focus of this paper, a proof of this result is postponed to forthcoming work of the first author.

\section{Measure theoretic properties}\label{sec:ergodic}

So far, we have not made any use of the probability vectors $\boldsymbol p_i$. One might wonder why we have introduced them, and why one is justified in calling these substitutions `random'. The reason is that in some cases, e.g. when one studies ergodic or spectral properties \cite{bss}, the probabilities play a vital role. Here, we focus on some aspects of the ergodic part.

It is well-known that, in the case of a primitive deterministic substitution $\dsub$, the frequency of each element of $\mc L_{\dsub}$ exists. Furthermore, it is encoded in the statistically normalised right Perron--Frobenius eigenvector of the associated substitution matrix; see \cite[Sec. 5.4]{qu}. The idea is to introduce an \emph{induced substitution} $\dsub_{\ell}$, which is a primitive substitution that acts on the set of elements of $\mc L_{\dsub}$ of length $\ell$. Such an induced substitution is also called \emph{right-collared} in the literature \cite{gm}, and a \emph{two-sided collared} version also exists.

Our aim is a generalisation of Queffelec's method to the random situation. This was first investigated by Moll \cite[Ch. 4]{mo} based on the example of the random noble means substitutions, which are generalisations of the random Fibonacci substitution mentioned above. Unfortunately, Moll's proof of ergodicity contains several errors. In light of this, we assume here that the measures we are going to construct are ergodic and we postpone a corrected proof of ergodicity to forthcoming work \cite{grs}. Using this result, we are able to characterise the uniquely and strictly ergodic dynamical systems that arise from random substitutions.

Next, we are going to introduce a substitution $\vartheta_{\ell}$ that acts on the alphabet $\mc A_{\ell}$ of $\vartheta$-legal words of length $\ell$. We denote by $\mc A_{\ell}^*$ the set of all finite words with respect to the alphabet $\mc A_{\ell}$. When we want to generalise Queffelec's method, we have to take into account the (possibly) different lengths of the images $\vartheta(a_i)$ and the realisation probabilities of subwords in the image of some word $w\in\mc A_{\ell}$ under $\vartheta$, since we work with random substitutions.  

\begin{definition}
Let $\ell\in\N$ and $\vartheta$ be a primitive random substitution. Then, we refer to
\[
\vartheta_{\ell} \colon \mc A_{\ell} \to \mc A_{\ell}^+
\]
as the \emph{induced random substitution} defined by
\[
w^{(i)} \mapsto 
  \begin{cases}
  u^{(i,1)} := \left( v_{[k,k+\ell-1]}^{(i,1)} \right)_{0\leq k\leq |\vartheta(w_0^{(i)})|-1}, & \text{ with probability } p_{i1} \\
 \quad \quad \quad \quad \quad \quad \vdots& \quad \quad \quad \quad \quad \vdots \\
  u^{(i,n_i)}:= \left( v_{[k,k+\ell-1]}^{(i,n_i)} \right)_{0\leq k\leq |\vartheta(w_0^{(i)})|-1}, & \text{ with probability } p_{in_i} ,
  \end{cases}
\]
where $w^{(i)}\in \mc A_{\ell}$ and $v^{(i,j)}$ is an image of $w^{(i)}$ under $\vartheta$  with probability $p_{ij}$. Here, $|\vartheta(w_0^{(i)})|$ is the length of the corresponding realisation of $\vartheta(w_0^{(i)})$.

The \emph{induced random substitution matrix} is given by $M_{\ell} := M_{\vartheta_{\ell}}$.
\end{definition}

\begin{example}
Consider the random Fibonacci substitution $\sigma$ from above. Let $\ell=2$. In this situation, we have $\mc A_2=\{aa,ab,ba,bb\}$, and we obtain the induced substitution $\vartheta_2$ given by
\[
(aa) \mapsto 
\begin{cases}
(ab)(ba) & \text{ with probability } q^2 \\
(ab)(bb) & \text{ with probability } qp  \\
(ba)(aa) & \text{ with probability } pq  \\
(ba)(ab) & \text{ with probability } p^2
\end{cases}, \quad  \quad 
(ab) \mapsto
\begin{cases}
(ab)(ba) & \text{ with probability } q \\
(ba)(aa) & \text{ with probability } p
\end{cases}
\]
\[
(ba) \mapsto
\begin{cases}
(aa) & \text{ with probability } q  \\
(ab) & \text{ with probability } p
\end{cases}, \quad  \quad
(bb) \mapsto (aa) \text{ with probability } 1.
\]
Hence, we get
\[
M_2 = 
\begin{bmatrix}
pq & p & q & 1 \\
1-pq & q & p & 0 \\
1-pq & 1 & 0 & 0 \\
pq & 0 & 0 & 0
\end{bmatrix}.
\]
This matrix is primitive. The Perron--Frobenius eigenvalue is $\tau$, where $\tau$ denotes the golden ratio, and the associated right eigenvector is 
\begin{equation} \label{Eq:R}
R=\frac{1}{\tau+2\tau^2(1-p+p^2)+p(1-p)} \cdot 
\begin{bmatrix}
\tau  \\
\tau^2(1-p+p^2) \\
\tau^2(1-p+p^2)\\
 p(1-p)
\end{bmatrix}
\end{equation}
\end{example}
\begin{remark}
Notice that $\vartheta_2$ admits the word $(ab)(ab)$ in its language because
\[
(ab)(ab) \blacktriangleleft \vartheta((ba)(aa)) \overset{\bullet}{=} \vartheta^2((ab)).
\]
This highlights a key difference between induced substitutions for deterministic substitutions as compared to random substitutions. For a primitive deterministic substitution $\dsub$, the induced substitution $\dsub_{\ell}$ gives rise to a topologically conjugate subshift---that is, $X_\dsub \simeq X_{\dsub_{\ell}}$. The topological conjugacy is given by a sliding block code of length $\ell$ which maps the letters in an element of $X_\dsub$ to the length $\ell$ word to their immediate right. The inverse is given by the forgetful map. In the case of RS-subshifts, as the above example suggests, such a sliding block code will in general not be surjective; there is no $\vartheta$-legal word, under the usual sliding block code, which will give rise to the $\vartheta_2$-legal word $(ab)(ab)$. In general, the most we can say is that the usual sliding block code gives an embedding of subshifts $e \colon X_\vartheta \hookrightarrow X_{\vartheta_{\ell}}$ and the usual forgetful map gives a topological factor map $f \colon X_{\vartheta_{\ell}} \twoheadrightarrow X_\vartheta$. The map $e$ is also a right inverse of $f$, giving $f \circ e = \operatorname{id}_{X_{\vartheta}}$. 
\end{remark}

As in the case for deterministic substitutions, an induced substitution for a primitive random substitution is also primitive.
\begin{prop}
Let $\ell$ be a positive integer and let $\vartheta$ be a random substitution. If $\vartheta$ is primitive then $\vartheta_\ell$ is primitive.
\end{prop}

\begin{proof}
The proof is similar to the usual proof for collared substitutions \cite{ap}.

It is enough to show that for any $\ell \in \N$, there exists a $k_\ell$ such that for every pair of words $u,v \in \mc L^\ell_\vartheta$ we have $u \blacktriangleleft \vartheta^{k_\ell}(v)$. By primitivity of $\vartheta$, let $k$ be such that for every $a_i, a_j \in \mc A$, we have $a_i \blacktriangleleft \vartheta^k(a_j)$. Without loss of generality, suppose that $k = 1$ (otherwise replace $\vartheta$ with $\vartheta^k$). By primitivity, let $k_u$ and $a_u \in \mc A$ be such that $u \blacktriangleleft \vartheta^{k_u}(a_u)$. As $a_u \blacktriangleleft \vartheta(a_i)$ for all $a_i \in \mc A$, we have $u \blacktriangleleft \vartheta^{k_u}(\vartheta(a_i)) \overset{\bullet}{=} \vartheta^{k_u+1}(a_i)$. Then by induction, $u \blacktriangleleft \vartheta^{k_u+n}(a_i)$ for every $n \geq 1$ and every letter $a_i \in \mc A$. As there are only finitely many letters in $\mc A$ and only finitely many pairs of words $u,v \in \mc L^\ell_\vartheta$, there must then exist a finite $k_{\ell} = \max\{k_u \mid u \in \mc L^\ell_\vartheta\} + n_{\ell}$ such that for every $u,v \in \mc L^\ell_\vartheta$, $u \blacktriangleleft \vartheta^{k_{\ell}}(v_0) \blacktriangleleft \vartheta^{k_{\ell}}(v)$, as required.
\end{proof}

Since the matrix $M_{\ell}$ is primitive, we know from Perron--Frobenius theory that there is a unique positive right eigenvector $\boldsymbol R_{\ell}$ and a unique positive left eigenvector $\boldsymbol L_{\ell}$ corresponding to the maximal positive eigenvalue $\lambda_{\ell}$ such that
\[
\| \boldsymbol R_{\ell}\|_1 =1 \quad \text{ and } \quad \left\langle\boldsymbol L_{\ell}, \boldsymbol R_{\ell} \right\rangle = 1.
\]  
As in the deterministic case, we would like to interpret the entries of $\boldsymbol R_{\ell}$ as the frequencies of the legal words of length $\ell$. In order to do so, let us proceed as follows.

An open, closed and countable basis for the topology of $\mc A^{\Z}$ is given by the class $\mathfrak Z(\mc A^{\Z})$ of \emph{cylinder sets}
\[
\mc Z_k(v):= \{w\in \mc A^{\Z} \mid w_{[k,k+\ell-1]}=v\}
\]
for any $k\in\Z$ and $v\in\mc A$ of length $\ell$. We refer to \cite[Sec. 2]{bi}, \cite[Ch. 6]{lm} and \cite[Ch. 4]{qu} for general background. If $X \subseteq \mc A^{\Z}$ is a subshift, the class of cylinder sets $\mathfrak Z(X)$ is induced by $\mathfrak Z(\mc A^{\Z})$ via the subspace topology. Here, we get
\[
\mathfrak Z(X) := \{\mathcal Z \, \cap \, X \mid \mathcal Z \in \mathfrak Z(\mc A^{\Z})\}.
\]
Consider the RS-subshift $X_{\vartheta}$. It is well-known that $\mathfrak Z(X_{\vartheta})$ generates a Borel $\sigma$-algebra $\mathfrak B_{\vartheta}$ of $X_{\vartheta}$, compare \cite[Sec. 4.1]{bible}. In later sections it will be useful to consider subclasses of $\mathfrak Z(X_{\vartheta})$ consisting of those cylinder sets that contain the reference point $0$:
\[
\mathfrak Z_0(X_{\vartheta}) := \{\mathcal Z_k(v) \in \mathfrak Z(X_{\vartheta}) \mid -|v|+1 \leq k \leq 0 \} \, \cup\, \{X_{\vartheta}\}.
\]
Note that $\mathfrak Z_0(X_{\vartheta})$ also generates the $\sigma$-algebra $\mathfrak Z(X_{\vartheta})$. Let $v$ be a legal word of length $\ell$. Then, we define a measure $\mu$ on $\mathcal Z_k(v) \in \mathfrak Z_0(X_{\vartheta})$ by
\begin{equation} \label{Eq:char_mu}
\mu(\mathcal Z_k(v)) := \boldsymbol R_{\ell}(v),
\end{equation}
for any $k\in\Z$, where $\boldsymbol R_{\ell}(v)$ is the entry of the statistically normalised right Perron--Frobenius eigenvector of the primitive matrix $M_{\ell}$ corresponding to the word $v$. According to \cite[Sec. 5.4]{qu}, this is a consistent definition of a measure on $\mathfrak Z_0(X_{\vartheta})$ and there is an extension of $\mu$ to the Borel $\sigma$-algebra $\mathfrak B_{\vartheta}$ \cite[Cor. 2.4.9]{par}. Due to \cite[Prop. 2.5.1]{par}, this extension is unique and we will denote it again as $\mu$. Furthermore, $\mu$ is a probability measure on $X_{\vartheta}$ because for any $k \in \Z$ and $\ell\in\N$, we have
\[
\mu(X_{\vartheta}) = \mu \Big( \dot{\bigcup_{v \in \mc L_{\vartheta}^{\ell}}} \mc Z(v) \Big) = \sum_{v \in \mc L_{\vartheta}^{\ell}} \mu \big(\mc Z_k(v)\big) = \sum_{v \in \mc L_{\vartheta}^{\ell}} \boldsymbol R_{\ell}(v) = 1.
\]
Moreover, the $S$-invariance of $\mu$ is an immediate consequence of \cite[Thm. 1.1]{wal} because $S^{-1} \mc Z_{k-1}(v) = \mc Z_k(v)$ and
\[
\mu\big(S^{-1}\mc Z_k(v)\big) = \mu \big(\mc Z_{k+1}(v)\big) = \boldsymbol R_{\ell}(v) = \mu \big(\mc Z_k(v)\big).
\]

We will show in \cite{grs} that, for certain random substitutions, the measure $\mu$ is ergodic. This result implies that the relative frequencies of subwords of length $\ell$ can be computed by the statistically normalised right Perron--Frobenius eigenvector $\boldsymbol R_{\ell}$. To see this, consider the function $1_{\mc Z_t(v)}(x)$ with $v\in \mc L_{\vartheta}^{\ell}$, which is obviously integrable. From the ergodicity of $\mu$, we infer
\[
\lim_{n\to\infty} \frac{1}{n} \sum_{i=s}^{n-1+s} 1_{\mc Z_t(v)}(S^ix) = \int_{X_{\vartheta}} 1_{\mc Z_t(v)}\ d\mu = \mu(\mc Z_t(v)) = \boldsymbol R_{\ell}(v)
\]
for $\mu$-almost every $x\in X_{\vartheta}$.

The next question that arises is whether we can expect the system $(X_{\vartheta},S)$ to be uniquely ergodic, which happens in the deterministic situation \cite[Thm. 5.6]{qu}. It is obvious that different choices of the probability vectors $\boldsymbol p_i$ lead to the same RS-subshift $X_{\vartheta}$. Still, in most cases, they give rise to different measures $\mu$. Consequently, the system $(X_{\vartheta},S)$ is not uniquely ergodic in general. For example, consider the random Fibonacci substitution $\sigma$. If we choose $\boldsymbol p_1=\big(\frac{1}{2},\frac{1}{2}\big)$, then for $\mu$-almost all $x\in X_{\sigma}$ we obtain from Eq. \eqref{Eq:R}
\[
\boldsymbol R_{2}(bb) \approx 0.043.
\]
On the other hand, choosing $\boldsymbol p_1=(1,0)$, we get
\[
\boldsymbol R_{2}(bb) =0.
\]
However, we will now characterise the uniquely/strictly ergodic systems.

\begin{theorem} \label{thm:str_erg}
Let $\vartheta$ be a primitive random substitution. 
\begin{enumerate}
\item[(a)] The corresponding dynamical system $(X_{\vartheta},S)$ is uniquely ergodic only if the right Perron--Frobenius eigenvectors $\boldsymbol R_{\ell}$ are independent of the probability vectors $\boldsymbol p_i$ for every $\ell$ and every $i$.
\item[(b)] Additionally, let $\mu$ be ergodic. Then, $(X_{\vartheta},S)$ is uniquely ergodic if and only if the right Perron--Frobenius eigenvectors $\boldsymbol R_{\ell}$ are independent of the probability vectors $\boldsymbol p_i$ for every $\ell$ and every $i$.
\end{enumerate}
\end{theorem}
\begin{proof}
It is recalled that, by Oxtoby's theorem \cite[Prop. 4.4]{bible}, a subshift is uniquely ergodic if and only if the frequencies of all finite subwords exist uniformly for each element in the subshift.  \vspace*{3mm} \\
(a) Let us assume that $(X_{\vartheta},S)$ is uniquely ergodic. This means that there is only one $S$-invariant probability measure $\mu$. This is precisely the one we constructed above, see Eq.\ (\ref{Eq:char_mu}). Consequently, it is the same for every choice of the probability vectors $\boldsymbol p_i$. Due to Eq.\ (\ref{Eq:char_mu}), the vectors $\boldsymbol R_{\ell}$ are independent of $\boldsymbol p_i$.    \vspace*{3mm} \\
(b) On the other hand, let us now assume that the vectors $\boldsymbol R_{\ell}$ are independent of the probabilities $\boldsymbol p_i$. We start with the case $\ell=1$ and set $\boldsymbol R:=\boldsymbol R_1$. As the substitution matrix $M$ of $\vartheta$ is given by $\big(\sum_{q=1}^{k_j} p_{jq} |w^{(j,q)}|_{a_i} \big)_{1\leq i,j\leq n}$, we obtain 
\[
\begin{bmatrix}
  \sum_{j=1}^n \sum_{q=1}^{k_j} p_{jq} |w^{(j,q)}|_{a_1} R_j \\
  \vdots \\
  \sum_{j=1}^n \sum_{q=1}^{k_j} p_{jq} |w^{(j,q)}|_{a_n} R_j
\end{bmatrix} = M\boldsymbol{R} = \lambda 
\begin{bmatrix}
R_1 \\
\vdots \\
R_n
\end{bmatrix},
\]
where $\lambda$ is the Perron--Frobenius eigenvalue of $M$. Thus, we have 
\[
\lambda = \sum_{j=1}^n \sum_{q=1}^{k_j} p_{jq}|w^{(j,q)}|_{a_i} \, \frac{R_j}{R_i}
\]
for all $i\in\{1,\ldots,n\}$ and therefore
\[
\sum_{j=1}^n \sum_{q=1}^{k_j}p_{jq}|w^{(j,q)}|_{a_{i_1}}\, \frac{R_j}{R_{i_1}} = 
\sum_{j=1}^n \sum_{q=1}^{k_j}p_{jq}|w^{(j,q)}|_{a_{i_2}}\, \frac{R_j}{R_{i_2}}
\]
for all $i_1,i_2\in\{1,\ldots,n\}$. Now, choose $j_1\in\{1,\ldots,n\}$ and fix $\boldsymbol p_j$ for all $j\in\{1,\ldots,n\}$ with $j\neq j_1$. Then, there are $c_1,c_2>0$, which are independent of $\boldsymbol p_{j_1}$ such that
\[
\sum_{q=1}^{k_{j_1}} p_{j_1q}|w^{(j_1,q)}|_{a_{i_1}}\, \frac{R_{j_1}}{R_{i_1}} + c_1 =
\sum_{q=1}^{k_{j_1}} p_{j_1q}|w^{(j_1,q)}|_{a_{i_2}}\, \frac{R_{j_1}}{R_{i_2}} + c_2. 
\] 
This is equivalent to
\[
\sum_{q=1}^{k_{j_1}} p_{j_1q} \left( |w^{(j_1,q)}|_{a_{i_1}}\, \frac{R_{j_1}}{R_{i_1}} - 
|w^{(j_1,q)}|_{a_{i_2}}\, \frac{R_{j_1}}{R_{i_2}} \right) = c_3,
\]
where $c_3:=c_2-c_1$, for every choice of $\boldsymbol p_{j_1}$. Next, choose $\boldsymbol p_{j_1} = e_{q_1}$ respectively $\boldsymbol p_{j_1} = e_{q_2}$ for some $q_1,q_2\in\{1,\ldots,k_{j_1}\}$.
We obtain
\[
|w^{(j_1,q_1)}|_{a_{i_1}}\, \frac{R_{j_1}}{R_{i_1}} - |w^{(j_1,q_1)}|_{a_{i_2}}\, \frac{R_{j_1}}{R_{i_2}} =
|w^{(j_1,q_2)}|_{a_{i_1}}\, \frac{R_{j_1}}{R_{i_1}} - |w^{(j_1,q_2)}|_{a_{i_2}}\, \frac{R_{j_1}}{R_{i_2}},
\]
which is equivalent to
\begin{equation}  \label{Eq:ratio}
|w^{(j_1,q_1)}|_{a_{i_1}} - |w^{(j_1,q_2)}|_{a_{i_1}} = 
\left( |w^{(j_1,q_1)}|_{a_{i_2}} - |w^{(j_1,q_2)}|_{a_{i_2}} \right) \, \frac{R_{i_1}}{R_{i_2}}.
\end{equation}
Now, Eq.\ (\ref{Eq:ratio}) tells us the following. If $w^{(j_1,q_1)}$ and $w^{(j_1,q_2)}$ are two different realisations of $\vartheta(a_{j_1})$ with $|w^{(j_1,q_1)}|_{a_{i_1}} = |w^{(j_1,q_2)}|_{a_{i_1}}$ for some letter $a_i$, then $|w^{(j_1,q_1)}|_{a_{i_2}} = |w^{(j_1,q_2)}|_{a_{i_2}}$ for every other letter $a_{i_2}$. On the other hand, if $|w^{(j_1,q_1)}|_{a_{i_1}} < |w^{(j_1,q_2)}|_{a_{i_1}}$, then $|w^{(j_1,q_1)}|_{a_{i_2}} < |w^{(j_1,q_2)}|_{a_{i_2}}$ for every other letter $a_{i_2}$ and the ratio 
\[
\frac{|w^{(j_1,q_1)}|_{a_{i_1}} - |w^{(j_1,q_2)}|_{a_{i_1}}}{ |w^{(j_1,q_1)}|_{a_{i_2}} - |w^{(j_1,q_2)}|_{a_{i_2}}}
\]
is given by $\frac{R_{i_1}}{R_{i_2}}$. Hence, the letter-frequencies exist uniformly for every element of $X_{\vartheta}$.

Since, for general $\ell$, we can argue analogously ($\vartheta_{\ell}$ is primitive for all $\ell\in\N$), the claim follows. 
\end{proof}

In general, uniquely ergodic systems are not minimal. However, in the case of primitive random substitutions, we obtain the following corollary.

\begin{coro}
Let $\vartheta$ be a primitive random substitution. If $(X_{\vartheta},S)$ is uniquely ergodic, it is also strictly ergodic. 
\end{coro}
\begin{proof}
Since $\vartheta$ is primitive and $(X_{\vartheta},S)$ is uniquely ergodic, we know that the relative word frequencies of finite words exist uniformly, for every element of $X_{\vartheta}$. Moreover, they are encoded in the entries of the vectors $\boldsymbol R_{\ell}$. But $\vartheta_{\ell}$ is primitive for all $\ell\ge1$, which implies that all entries of $\boldsymbol R_{\ell}$ are positive for all $\ell\ge1$. Due to Oxtoby's theorem, the system $(X_{\vartheta},S)$ is strictly ergodic. 
\end{proof}

Now, we can make use of the right eigenvectors $\boldsymbol R_{\ell}$ to give a characterisation of the minimality of $(X_{\vartheta},S)$ when $\vartheta$ is primitive and $\mu$ is ergodic.

\begin{prop}
Let $\vartheta$ be a primitive random substitution and suppose that the measure $\mu$ is ergodic. The RS-subshift $(X_{\vartheta},S)$ is minimal if and only if the vectors $\boldsymbol R_{\ell}$ do not depend on $\boldsymbol p_i$ for all $\ell\ge 1$ and all $i\in\{1,\ldots,k_i\}$.
\end{prop}
\begin{proof}
Note that, by Proposition \ref{prop:lin}, the set $\operatorname{Lin}(X_\vartheta)$ of linearly repetitive elements of $X_\vartheta$ is dense. If $X_\vartheta$ is minimal, then $\operatorname{Lin}(X_\vartheta) = X_\vartheta$. By construction (in the proof of Proposition \ref{prop:lin}), $X_\vartheta$ is equal to a minimal subshift $X_\dsub$, where $\dsub$ is a primitive deterministic substitution. In particular, it is well-known that the relative frequency of a word is independent of the element chosen from a minimal deterministic substitution subshift.

Assume that $\boldsymbol R_{\ell}$ does depend on $\boldsymbol p_i$ for some $\ell$ and some $i$. Then, there are $k,j\in\N$ such that the $k$-th entry of $\boldsymbol R_{\ell}$, let us call it $\boldsymbol R_{\ell,k}$, depends on some $p_{ij}$. This means that the frequency of the word $u$ which corresponds to $\boldsymbol R_{\ell,k}$ depends on $p_{ij}$. But this implies that there are $w,w'\in X_{\vartheta}$ such that
\[
\freq_u(w) \neq \freq_u (w').
\]
Hence, the system $(X_{\vartheta},S)$ cannot be minimal.

Now for the converse, suppose that $\boldsymbol R_{\ell}$ is independent of $\boldsymbol p_i$. Then, $(X_{\vartheta},S)$ is strictly ergodic, hence minimal, by Theorem \ref{thm:str_erg} and the ergodicity of $\mu$.
\end{proof}

We finish this section with the following observations which highlights the interplay (and differences) between the topological notion of dense subsets and the measure theoretic notion of having full measure. Compare and contrast this result with the topological results from Section \ref{sec:dynamics}.

\begin{prop}
Let $\vartheta$ be a primitive random substitution. Assume that $(X_{\vartheta},S,\mu)$ is ergodic, where $\mu$ is given by Eq.\ (\ref{Eq:char_mu}). Then,
\begin{enumerate}
\item[(i)] the set $\operatorname{Dense}(X_\vartheta)$ of elements with dense orbit has measure $1$.
\item[(ii)] the set $\operatorname{Rep}(X_\vartheta)$ of repetitive elements has measure $1$ if $X_{\vartheta}$ is minimal and measure $0$ otherwise.
\item[(iii)] the set $\operatorname{Lin}(X_\vartheta)$ of linearly repetitive elements has measure $1$ if $X_{\vartheta}$ is minimal and measure $0$ otherwise.
\item[(iv)] the set $\operatorname{Per}(X_\vartheta)$ of elements with periodic orbit has measure $1$ if $X_\vartheta$ is finite and measure $0$ otherwise.
\end{enumerate}
\end{prop}
\begin{proof}
(i) Let $\operatorname{Dense}(X,f) \subset X$ be the set of points in $X$ with dense orbit under the homeomorphism $f \colon X \to X$. We make use of a well-known result \cite[Thm. 1.7]{wal} which says that $\operatorname{Dense}(X,f)$ has full measure in an ergodic dynamical system $(X,f,\mu)$ whose non-empty open sets all have positive measure. The only thing which remains to be shown is that $\mu(U)>0$ for every non-empty open set $U$. We mentioned earlier that $\mathfrak Z(X_{\vartheta})$ is an open basis for the topology of $\mc A^{\Z}$. Hence, it suffices to show $\mu(\mc Z)>0$ for every $\mc Z\in \mathfrak Z(X_{\vartheta})$. But this follows from Eq.\ (\ref{Eq:char_mu}) and the Perron--Frobenius theorem because $\vartheta$ is primitive and so is $\vartheta_\ell$ for all $\ell$.

(ii) It is well-known that if $X_\vartheta$ is minimal, then every element is repetitive and so
\[\mu(\operatorname{Rep}(X_\vartheta)) = \mu(X_\vartheta) = 1.\]
If $X_\vartheta$ is non-minimal, then for $x \in \operatorname{Rep}(X_\vartheta)$, the orbit closure $\overline{\mathcal{O}x}$ is a minimal subset of $X_\vartheta$ and so $\mathcal{O}x$ cannot be dense in $X_\vartheta$. It follows that $\operatorname{Rep}(X_\vartheta) \cap \operatorname{Dense}(X_\vartheta) = \emptyset$. So, by part (i), countable additivity and $\mu(E) \leq 1$ gives
\[
\mu(\operatorname{Rep}(X_\vartheta)) = \mu(\operatorname{Rep}(X_\vartheta)) + \mu(\operatorname{Dense}(X_\vartheta)) - 1 = \mu(\operatorname{Rep}(X_\vartheta) \cup \operatorname{Dense}(X_\vartheta)) - 1 \leq 0.
\]

(iii) The same proof as for part (ii) follows for $\operatorname{Lin}(X_\vartheta)$ using the fact that repetitivity and linear repetitivity are equivalent for minimal subshifts associated with primitive deterministic substitutions and the fact that $\operatorname{Lin}(X_\vartheta) \subseteq \operatorname{Rep}(X_\vartheta)$.

(iv) If $X_\vartheta$ is finite, then every point is periodic so $\mu(\operatorname{Per}(X_\vartheta)) = \mu(X_\vartheta) = 1$. Now suppose that $X_\vartheta$ is infinite. If $\operatorname{Per}(X_\vartheta) = \emptyset$, then $\mu(\operatorname{Per}(X_\vartheta)) = 0$. If $\operatorname{Per}(X_\vartheta) \neq \emptyset$, then as $X_\vartheta$ is infinite, the RS-subshift cannot be minimal as it has a proper closed invariant subset given by a periodic orbit. It follows from part (iii) that $\mu(\operatorname{Lin}(X_\vartheta)) = 0$. The set $\operatorname{Per}(X_\vartheta)$ is a subset of $\operatorname{Lin}(X_\vartheta)$ and so $\mu(\operatorname{Per}(X_\vartheta)) \leq \mu(\operatorname{Lin}(X_\vartheta)) = 0$.
\end{proof}

\section{Topological entropy}\label{sec:entropy}
The aim of this section is to examine the amount of disorder inherent in the random substitution $\vartheta$. To do so, we refer to the map $C \colon \N \to \N$, assigning to the natural number $\ell$ the number of $\vartheta$-legal words of this length, as the \emph{complexity function} of $\vartheta$. If $X$ is a subshift, then we similarly let $C(\ell) = |\mc L^\ell(X)|$.

\begin{definition}
Let $X$ be a subshift and let $C$ be the complexity function of $X$. Then, we let
\begin{equation} \label{Eq:entropy}
 h_{\text{top}}(X) := \lim_{\ell\to\infty} \frac{\log\big( C(\ell)\big)}{\ell}
\end{equation}
denote the \emph{topological entropy} of $X$.
\end{definition}

Obviously, the complexity function meets $C(k+\ell) \leq C(k)C(\ell)$ for all $k,\ell\in \N$. The existence of the limit in Eq.\ (\ref{Eq:entropy}) in this setting is a well-known result; the proof is based on Fekete's Lemma.

One might be inclined to think that every non-deterministic RS-subshift has positive entropy. However, this is not the case. Consider the random substitution 
\[
\vartheta \colon a \mapsto \{ab, abab\}, b \mapsto \{ab\}.
\]
As previously remarked, the corresponding RS-subshift consists of the two periodic bi-infinite sequences 
\[
\cdots abab.abab\cdots\quad \text{ and } \quad \cdots baba.baba \cdots,
\] 
thus, $h_{\text{top}}(X_\vartheta) = 0$. In fact, this happens because one of the words in the set of realisations of $\vartheta(a)$ is redundant. This means that omitting either $ab$ or $abab$ will not change the subshift and, moreover, gives rise to the subshift of a deterministic substitution. The problem is that the first two letters and the last two letters of the second realisation of the image of $a$ coincide with the first realisation of the image of $a$. This example motivates a useful definition.

Recall that $u$ is an \emph{affix} of $v$ if $u$ is either a prefix or a suffix of $v$. That is, $u = v_{[0,|u|-1]}$ or $u = v_{[|v|-|u|,|v|-1]}$. If $u$ is both a prefix and a suffix of $v$ then we call $u$ a \emph{strong affix} of $v$.
\begin{definition}
If $a \in \mc A$ is such that there exist realisations $u,v \overset{\bullet}{=} \vartheta(a)$ with $|u| \leq |v|$ and $u$ is not a strong affix of $v$ then we say that $a$ \emph{admits a splitting pair} for $\vartheta$.
\end{definition}
The intuition is that whenever $a$ admits a splitting pair, then when substituting a word $u_0$ containing $a$, there are at least two possible distinct realisations of $\vartheta(u_0)$ and so a `splitting' or `branching' occurs in the tree of iterated substitutions of $u_0$. The condition regarding strong affixes ensures that we cannot accidentally double count, as emphasised in the proof of the following theorem.

Note that, trivially, no letter admits a splitting pair for a deterministic substitution.
\begin{theorem}\label{thm:entropy}
Let $\vartheta$ be a random substitution. If there is $w \in X_{\vartheta}$ such that $a$ appears in $w$ with positive letter-frequency $\nu(a)$ and $a$ admits a splitting pair for $\vartheta$, then the system $(X_{\vartheta},S)$ has positive entropy. 
\end{theorem}
\begin{proof}
Let $u$ and $v$ be a splitting pair admitted by $a_i$ with lengths  $|u|=k_1$ and $|v|=k_2$ and without loss of generality assume that $k_1 \leq k_2$. Denote by $W_n$ the set of different realisations of $\vartheta(w_{[-n,n]})$. If $N\ge |\vartheta(b)|$ for every $b \in \mc A$ and every possible realisation of $\vartheta(b)$, one clearly has
\begin{equation} \label{Eq:entropy2}
\limsup_{n \to \infty} \frac{\log|W_n|}{N\, (2n+1) } \leq \limsup_{n\to\infty} \frac{\log|W_n|}{|\vartheta(w_{[-n,n]})|} \leq h_{\text{top}}(X_\vartheta).
\end{equation}
Now, choose for every letter $b \neq a$ a fixed realisation $u_b$ of $\vartheta(v)$. The letter $a$ is alternately mapped to $u$ and $v$. In that case, given any word $u_0 \triangleleft w$ which does not contain the letter $a$, we obtain that the word $a u_0 a$ is mapped to either
\[
u \vartheta(u_0) v \quad \text{ or } \quad v \vartheta(u_0) u.
\]
These two words have the same length, and they are different by assumption because $u$ is not a strong affix of $v$. Suppose $|w_{[-n,n]}|_a$ is even. Let us partition the word $w_{[-n,n]}$ into alternating blocks $U_i$ which contain no $a$s and blocks $V_i$ which contain exactly two $a$s on their boundary. If $|w_{[-n,n]}|_a$ is odd, then we include a leftover block $\tilde{U}$ at the end of the word which contains exactly one $a$ (or is equal to $U_{\lfloor\frac{1}{2}|w_{[-n,n]}|_a\rfloor}$ if $|w_{[-n,n]}|_a$ is even). So if
\[
w_{[-n,n]} = u_1au_2au_3au_4au_5\cdots au_{k-1}au_k
\]
where $|u_j|_a = 0$ for all $j<k$ and $|u_k|_a \leq 1$, then the blocks $U_i$ are given by $U_i = u_{2i-1}$ and the blocks $V_i$ are given by $V_i = au_{2i}a$. So
\[
\begin{array}{rcl}
w_{[-n,n]} 	& = & \overbrace{u_1}^{U_1}\overbrace{au_2a}^{V_1}\overbrace{u_3}^{U_2}\overbrace{au_4a}^{V_2}\overbrace{u_5}^{U_3} \cdots \overbrace{au_{k-1}a}^{V_{\lfloor\frac{1}{2}|w_{[-n,n]}|_a\rfloor}}\overbrace{u_k}^{\tilde{U}}\\
			& = & U_1 V_1 U_2 V_2 U_3 \cdots V_{\lfloor\frac{1}{2}|w_{[-n,n]}|_a\rfloor} \tilde{U}
\end{array}
\]
Since, by the above, every two consecutive $a$s give rise to at least two different words under substitution, hence we have at least two possible choices of image under substitution for every block $V_i$ in the partition of $w_{[-n,n]}$, then we obtain
\[
|W_n| \ge 2^{\lfloor\frac{1}{2}|w_{[-n,n]}|_a\rfloor}.
\]
Consequently,
\[
\limsup_{n\to\infty} \frac{\log|W_n|}{N\, (2n+1)} \ge \limsup_{n\to\infty} \frac{\lfloor\frac{1}{2}|w_{[-n,n]}|_a\rfloor}{(2n+1)}\cdot \frac{\log(2)}{N} = \frac{\nu(a)}{2N}\, \log(2)>0.
\]
Hence, together with Eq.\ (\ref{Eq:entropy2}), the claim follows.
\end{proof}

\begin{remark}
It is not necessary that the letter $a$ has positive letter-frequency. The proof shows that the condition $\limsup_{n\to\infty} \frac{\lfloor\frac{1}{2}|w_{[-n,n]}|_{a}\rfloor}{(2n+1)} >0$ is sufficient. This condition is satisfied, for example, if $\vartheta$ is primitive. \exend 
\end{remark}

\begin{coro}
Let $\vartheta$ be a primitive random substitution with a letter admitting a splitting pair. Then $h_{\text{top}}(X_\vartheta) > 0$.
\end{coro}

\begin{remark}
It is sufficient if some power of $\vartheta$ has a letter admitting a splitting pair. Consider for example the primitive substitution
\[
 \vartheta \colon a\mapsto \{ab, abab \}, b \mapsto \{abb\}.
\]
This substitution $\vartheta$ has no letter admitting a splitting pair, however both letters admit splitting pairs for the square $\vartheta^2$. Hence, the corresponding RS-subshift has positive entropy.  \exend
\end{remark}
We conjecture that for primitive substitutions, the converse of Theorem \ref{thm:entropy} is also true up to taking powers. That is, for a primitive random substitution $\vartheta$, we make the conjecture that the RS-subshift $X_\vartheta$ admits positive topological entropy if and only if there exists a power $k \geq 1$ and a letter $a \in \mc A$ such that $a$ admits a splitting pair for $\vartheta^k$.

\begin{example}
Consider the random Fibonacci substitution. By the previous theorem, we know that the corresponding RS-subshift has positive entropy because $a$ admits the splitting pair $ab$ and $ba$. The precise value is $\sum_{i=2}^{\infty} \frac{\log(i)}{\tau^{i+2}} \approx 0.444399$, which can be found in \cite[Ch. 3]{mo}.
\end{example}

\section{Examples and open questions}\label{sec:examples}
\begin{example}
Let $\mc A =\{0,1\}$ be a binary alphabet. Let $\mc A^\Z$ be the full shift on $\mc A$. We leave it as an exercise to the reader to show that if we define the random substitution $\vartheta$ by
\[
\vartheta \colon 0 \mapsto \{00,01,10,11\}, 1 \mapsto \{00,01,10,11\},
\]
then $\vartheta$ is primitive and $X_\vartheta = \mc A^\Z$. By writing the full shift in terms of a random substitution, this offers a novel proof for the well-known property that the full shift is topologically transitive and has a dense set of periodic points via a simple application of Propositions \ref{prop:top_tran} and \ref{prop:periodic} and the observation that the periodic element $\cdots 000.000 \cdots$ is in the full shift.
\end{example}
\begin{example}
We can give another representation of the full shift on two letters as an RS-subshift which is in some ways more appealing, although the substitution is an example where the set of realisations of a substituted letter is infinite. Due to this, one should be careful when applying the machinery that has been set up in the previous sections, as not all methods of proof carry over in the infinite-image case. For instance, the proof of Theorem \ref{thm:entropy} relies on the existence of a number $N \geq |\vartheta(a_j)|$ for every $a_j \in \mc A$, which this particular example does not satisfy (even though the conclusion of positive entropy still holds).

Let the random substitution $\vartheta$ on the alphabet $\mc A = \{a,b\}$ be given by
\[
\vartheta \colon \left\{
\begin{array}{rll}
  a & \mapsto ba^n, & \text{ with probabilty } \dfrac{1}{2^{n+1}}, \quad n = 0, 1, \ldots \\
  b & \mapsto ab^m, & \text{ with probabilty } \dfrac{1}{2^{m+1}}, \quad m = 0, 1, \ldots
\end{array}
\right. .
\]
The substitution $\vartheta$ is primitive. Any word $u \in \mc A^+$ can be written as $u = a^{n_1}b^{m_1}a^{n_2}b^{m_2}\cdots a^{n_k}b^{m_k}$ with $n_i,m_i \geq 1$ for all $1 \leq i \leq k$ except possibly $n_1, m_k = 0$. The word $u$ can then be rewritten as
\[
u = a^{n_1-1}(ab^{m_1-1})(ba^{n_2-1})(ab^{m_2-1}) \cdots (ba^{n_k-1})(ab^{m_k})
\]
where we obviously regroup letters if $n_1 = 0$. This then shows that $u$ is a subword of a realisation of $\vartheta((ab)^k)$. As $(ab)^k \blacktriangleleft \vartheta(ba^k) \blacktriangleleft \vartheta^2(a)$, it follows that $u \blacktriangleleft \vartheta^3(a)$ for all $u \in \mc A^\ast$ and so $X_\vartheta = \mc A^\Z$. Using the equality $\sum_{n=0}^\infty n/2^{n+1} = 1$, it is clear that the substitution matrix for $\vartheta$ is just the matrix
\[
M_\vartheta = 
\begin{bmatrix}
1&1\\
1&1
\end{bmatrix}
\]
and so we have Perron--Frobenius eigenvalue $\lambda_\vartheta = 2$ and normalised left and right eigenvectors $\boldsymbol L = [1,1]$ and $\boldsymbol R = [1/2,1/2]^T$.
For ease of notation, let us write the two-letter word $xy$ as $x_y$ in its `right-collared' form. The induced substitution $\vartheta_2$ on two-letter words is given by
\[
\vartheta_2 \colon \left\{
\begin{array}{rl}
  a_a & \mapsto \begin{cases}b_b, & \text{ with probabilty } \dfrac{1}{2} \\ b_a a_a^{n-1} a_b, & \text{ with probabilty } \dfrac{1}{2^{n+1}}, \quad n = 1, 2, \ldots\end{cases} \\
  a_b & \mapsto b_a a_a^n, \quad \quad \quad \:\:\:  \text{ with probabilty } \dfrac{1}{2^{n+1}}, \quad n = 0, 1, \ldots \\
  b_a & \mapsto a_b b_b^m, \quad \quad \quad \:\:\: \text{ with probabilty } \dfrac{1}{2^{m+1}}, \quad m = 0, 1, \ldots \\
  b_b & \mapsto \begin{cases}a_a, & \text{ with probabilty } \dfrac{1}{2} \\ a_b b_b^{m-1} b_a, & \text{ with probabilty } \dfrac{1}{2^{m+1}}, \quad m = 1, 2, \ldots\end{cases}
\end{array}
\right. .
\]
with associated induced substitution matrix given by
\[
M_{\vartheta_2} = 
\begin{bmatrix}
1/2 & 1 & 0 & 1/2\\
1/2 & 0 & 1 & 1/2\\
1/2 & 1 & 0 & 1/2\\
1/2 & 0 & 1 & 1/2
\end{bmatrix}.
\]
As expected, the Perron--Frobenius eigenvalue of $M_{\vartheta_2}$ is still $\lambda_{\vartheta_2} = 2$ and the left and right eigenvectors are given by $\boldsymbol L = [1,1,1,1]$ and $\boldsymbol R = [1/4,1/4,1/4,1/4]^T$.

The above suggests that the measure $\mu$ associated with $\vartheta$ most likely coincides with the uniform Bernoulli measure for the full shift.
\end{example}

\begin{example}\label{ex:golden}
Let $\mc A =\{0,1\}$. Let $\mc F = \{11\}$ be the set of forbidden words for the shift of finite type 
\[
X_{\mc F} = \{w \in \mc A^\Z \mid u \in \mc F \implies u \notin \mc L(w)\}.
\]
The subshift $X_{\mc F}$ is often called the \emph{golden shift}. We claim that the RS-subshift $X_\vartheta$ associated with the primitive random substitution given by
\[
\vartheta \colon 0 \mapsto \{010,0\}, 1 \mapsto \{01,1\}
\]
is equal to the shift of finite type $X_{\mc F}$. Clearly the word $11$ is not in the language $\mc L_\vartheta$ because there is no letter $a$ such that $11 \blacktriangleleft \vartheta(a)$ and the only two letter word $ab$ such that $11 \blacktriangleleft \vartheta(ab)$ is $ab = 11$ itself. So we clearly have $X_\vartheta \subseteq X_{\mc F}$.

For the other inclusion, let $u \in \mc L(X_{\mc F})$. Suppose that $|u|_1 = m$. We can form the word $(01)^{m+1}$ as a subword of $\vartheta^{m+1}(0)$ by always realising $\vartheta(0)$ as $010$ and $\vartheta(1)$ as $1$. We may then generate $u$ from $(01)^{m+1}$ by `padding out' enough $0$s in between the $1$s. For instance, to generate the word $0100100010$ we substitute
\[
(01)^4 = 010\dot{1}0\dot{1}01 	 \mapsto 0\;1\;0\;01\;0\;0\dot{1}\;0\;1 \mapsto 0\;1\;0\;0\;1\;0\;0\;01\;0\;1,
\]
where a dot above a letter indicates that it will be substituted non-trivially. This word contains $0100100010$ as a subword. In this way, every element of the language $\mc L(X_{\mc F})$ is legal for $\vartheta$ and so $X_{\mc F} \subseteq X_\vartheta$.
\end{example}

The above result can be generalised using a more systematic method. It can be shown that every topologically transitive shift of finite type can be realised up to topological conjugacy as a primitive RS-subshift \cite{grs}.

\begin{example}
Let $\mc A = \{0,1\}$ be a binary alphabet. Let
\[
\vartheta_{PD} \colon 0 \mapsto \{01,10\}, 1 \mapsto \{00\}
\]
be the so-called \emph{random period doubling substitution}, see \cite{bss,htw}. The substitution matrix is given by
\[
M_{\vartheta_{PD}} =
\begin{bmatrix}
1&2\\
1&0
\end{bmatrix}
\]
which is primitive with Perron--Frobenius eigenvalue $\lambda_{\vartheta_{PD}} = 2$ and normalised left and right eigenvectors $\boldsymbol L = [\frac{1}{3},\frac{1}{3}]$ and $\boldsymbol R = [2,1]$. This implies that the ratio of $0$s to $1$s in any substituted word will be $2:1$. It follows that if any periodic element $w$ of $X_{\vartheta_{PD}}$ exists, the periods of $x$ must be divisible by $3$. Indeed, we find that the $3$-periodic element $\cdots 001\;001.001\;001\cdots$ is $\vartheta_{PD}$-legal. The word $001$ appears as a subword of $0010 \blacktriangleleft \vartheta_{PD}(10) \blacktriangleleft \vartheta_{PD}^2(0)$. From $001$ we can generate any power $(001)^k$ as a subword of $\vartheta_{PD}^{n_k}(001)$ for some $n_k$ by using the rule that we always substitute $0_0 \mapsto 10$, $0_1 \mapsto 01$ and $1 \mapsto 00$, where here we use the notation that $x_y$ is any appearance of $x$ immediately preceding a $y$. Thus, we have
\[
001 \mapsto 10\;01\;00 \mapsto 00\;10\;01\;00\;10\;01 \mapsto 10\;01\;00\;10\;01\;00\;10\;01\;00\;10\;01\;00 \mapsto \cdots
\]

The method for producing a periodic point in this way is rather ad hoc. It would be useful to have a method which works in general. Certainly, as in the case of the random Fibonacci substitution, a sufficient obstruction to periodic points existing in the RS-subshift is for the relative frequencies of a pair of letters (or words of a given length $\ell$) to be irrationally related for all elements of the subshift.

\begin{question}
Does there exist an effective method for determining if $\operatorname{Per}(X_\vartheta)$ is empty or non-empty?
\end{question}

By Proposition \ref{prop:periodic}, the periodic points of $X_{\vartheta_{PD}}$ form a proper dense subset (proper because not all elements are periodic). Although Theorem \ref{thm:entropy} tells us that $X_{\vartheta_{PD}}$ has positive entropy, we actually know more. For the specific case of $\vartheta_{PD}$ being the random period doubling substitution, the entropy $h_{\text{top}}(X_{\vartheta_{PD}})$ has been calculated in \cite{bss} to be
\[
h_{\text{top}}(X_{\vartheta_{PD}}) = \frac{2}{3}\log 2.
\]

Certain aspects of the RS-subshift $X_{\vartheta_{PD}}$ are similar to those of an irreducible shift of finite type. For instance $X_{\vartheta_{PD}}$ is a Cantor set by Proposition \ref{prop:cantor}, is topologically transitive by Proposition \ref{prop:top_tran}, contains a dense set of periodic points, and has topological entropy given by the logarithm of an algebraic number.

Recall that a subshift $X$ is \emph{topologically mixing}, or just \emph{mixing} if for all $u, v \in \mc L(X)$, there exists an $N \geq 0$ such that for all natural numbers $n \geq N$, there exists a word $w$ of length $|w| = n$ such that $uwv \in \mc L(X)$.
Recall that if $f \colon X \to Y$ is a factor map of dynamical systems, then if $X$ is mixing, so is $Y$. A shift of finite type is called \emph{primitive} if it is conjugate to an edge shift of a directed graph $G$ whose adjacency matrix $A_G$ is primitive. A shift of finite type is primitive if and only if it is mixing \cite[Prop. 4.5.10]{lm}, hence all factors of primitive shifts of finite type are mixing.

Let $u = v = bb$. This word is $\vartheta_{PD}$-legal. Moreover, $bb$ can only ever appear as a $\vartheta_{PD}$-legal subword of the word $abba$ with the only possible partition of this word into legal substituted words being $abba = (ab)(ba) \overset{\bullet}{=} \vartheta_{PD}(a)\vartheta_{PD}(a)$. Suppose there exists a word $w$ of odd length such that $uwv$ is legal. By the above, this means that the word $auwva$ must also be legal, and moreover this word can be partitioned exactly into legal substituted words because the end points of this word are uniquely partitioned as such. However, every substituted word has even length, and so the length of $auwva$ must also be even. This contradicts the assumption that the length of $w$ is odd. It follows that $X_{\vartheta_{PD}}$ is not mixing and so $X_{\vartheta_{PD}}$ cannot be a factor of a primitive shift of finite type.
\end{example}
It would be nice to extend the above argument to a general class of random substitutions which are not mixing. Clearly some RS-subshifts are mixing, as in the case of the golden shift and the full shift.
\begin{question}
Given a primitive random substitution $\vartheta$, can we determine necessary or sufficient conditions for $X_\vartheta$ to be topologically mixing?
\end{question}
As RS-subshifts have potentially non-trivial subsets of periodic points, it makes sense to try and characterise the structure of the periodic points of an RS-subshift. This is often studied via the \emph{Artin-Mazur zeta function} of the subshift. The zeta function $\zeta_{\vartheta}$ is defined by
\begin{equation} \label{Eq:zeta}
\zeta_\vartheta(z) = \exp\left(\sum_{n=1}^\infty |\operatorname{Fix}(S^n)|\frac{z^n}{n}\right)
\end{equation}
where $\operatorname{Fix}(S^n) = \{w \in X_\vartheta \mid S^n(w) = w\}$. The zeta function is an invariant of $X_\vartheta$ up to topological conjugacy. The zeta functions of shifts of finite type have been well studied \cite{lm} and so we can determine $\zeta_\vartheta$ for some of the examples considered in this section. In fact, the zeta function is just given as the reciprocal of the characteristic polynomial of a non-negative matrix associated to the shift. For instance, we know that the golden shift of Example \ref{ex:golden} has zeta function $\zeta_\vartheta(z) = (\det (1 - z[\begin{smallmatrix}1&1\\1&0\end{smallmatrix}]))^{-1} = (z^2 - z -1)^{-1}$. This relies on us knowing that the golden shift can be represented as a shift of finite type however.

\begin{example}
One can also describe some sofic shifts of non-finite type using random substitutions. Consider the random substitution 
\[
\vartheta \colon a \mapsto \{ab,ba\}, b \mapsto \{ab,ba\}.
\]
It is not difficult to see that $X_{\vartheta}$ is a sofic subshift. The right-resolving graph of $X_\vartheta$ is shown in Figure \ref{fig:sofic}---elements of $X_\vartheta$ are precisely those coded by bi-infinite directed paths in the right-resolving graph with the corresponding edge-labels.
\begin{figure}[h]
\centering

\begin{tikzpicture} 
  \SetGraphUnit{3}
  \Vertex{2}
  \WE(2){1}
  \EA(2){3}
  \Edge[label = $a$, color = red](1)(2)
  \Edge[label = $a$, color = red](2)(3)
  \Edge[label = $b$, color = blue](3)(2)
  \Edge[label = $b$, color = blue](2)(1)
  \tikzset{EdgeStyle/.append style = {bend left = 50}}
\end{tikzpicture}

\caption{The right-resolving graph representation of $X_\vartheta$ as a sofic shift.}

\label{fig:sofic}

\end{figure}
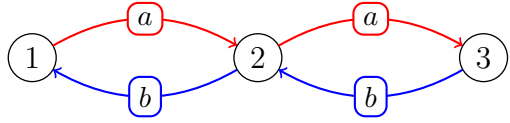
Applying standard techniques \cite[Thm. 6.4.8]{lm}, we find that the zeta function $\zeta_{\vartheta}$ is given by
\[
\zeta_{\vartheta}(z) = \frac{1-z^2}{1-2z^2} \, ,
\]
whence $X_{\vartheta}$ cannot be an SFT. However, using the explicit structure of $\vartheta$, there is an elementary method for determining $\zeta_{\vartheta}$. One quickly verifies that $|\operatorname{Fix}(S^n)| = 2^{k+1}-2$ if $n=2k$, and $0$ otherwise. This comes from the fact that every legal word of length $2k$ can appear as the periodic block in a periodic word of period $2k$. Such a word is either an exact concatenation of words of the form $\vartheta(a)$ so $u \overset{\bullet}{=} \vartheta(a)^k$ (of which there are $2^k$ choices of $\vartheta(a)$) or is the exact concatenation of words of the form $\vartheta(a)$ extended by either an $a$ or $b$ on the left and then by the other choice of $b$ or $a$ on the right so $u \overset{\bullet}{=} x\vartheta(a)^{k-1}y$ where $x \neq y$ (of which there are $2\cdot 2^{k-1} = 2^{k}$ choices). The only words of length $2k$ that can be formed in both ways are the words $(ab)^k$ and $(ba)^k$, as any appearance of the subword $aa$ or $bb$ uniquely determines the supertile structure. It follows that we have $2^k + 2^k - 2 = 2^{k+1}-2$ possible legal words of length $2k$ and hence $|\operatorname{Fix}(S^n)| = 2^{k+1}-2$.

Now, by Eq.\ (\ref{Eq:zeta}), we calculate
\[
\begin{array}{rcl}
\zeta_\vartheta(z)	& = & \exp \left(\sum_{k=1}^{\infty}(2^k-1) \frac{z^{2k}}{k}\right)\\
					& = & \exp \left(\sum_{k=1}^{\infty}\frac{(2z^2)^k}{k} - \sum_{k=1}^{\infty}\frac{(z^2)^k}{k}\right)\\
					& = & \exp \left(- \log (1-2z^2) + \log (1-z^2)\right) \\
					& = & \frac{1-z^2}{1-2z^2}.
\end{array}
\]
\end{example} 

We have also calculated, via ad hoc methods, the first few terms of the zeta function of the random period doubling substitution to be $\zeta_{\vartheta_{PD}}(z) = \exp(z^3 + 2z^6 + \frac{5}{3}z^9 + \cdots)$ but we presently have no method for determining arbitrarily large terms and are far from being able to present $\zeta_{\vartheta_{PD}}(z)$ as a closed form expression in $z$.
\begin{question}
Given a primitive random substitution $\vartheta$, does there exist an effective method for calculating its zeta function $\zeta_\vartheta$?
\end{question}
As suggested by the results of Section \ref{sec:entropy}, the topological entropy is also a useful invariant of RS-subshifts. Again, in the case that the RS-subshift is conjugate to a shift of finite type, there are well-known methods for calculating the topological entropy of the subshift, given in terms of the logarithm of the Perron--Frobenius eigenvalue of an integer matrix. Also, for specific examples, and small families of examples which are not shifts of finite type, we can give an explicit description of the topological entropy \cite{bss,gl,mo,nil,wi}. However, these methods do not obviously generalise.
\begin{question}
Given a primitive random substitution $\vartheta$, does there exist an effective method for calculating the topological entropy of its RS-subshift $h_{\text{top}}(X_\vartheta)$?
\end{question}

\section*{Acknowledgements}
The authors wish to thank Michael Baake, Chrizaldy Neil Manibo and Philipp Gohlke for helpful 
discussions. This work is supported by the German Research Foundation (DFG) via the Collaborative Research Centre (CRC 1283) through the faculty of Mathematics, 
Bielefeld University.

\end{document}